\documentclass[11pt]{amsart}

\usepackage{graphicx,color}
\usepackage{geometry}
\usepackage[all]{xy}
\usepackage{amssymb}
\usepackage{cite}

\newtheorem{introtheorem}{Theorem}

\swapnumbers
\newtheorem{theorem}{Theorem}[section]

\newtheorem{corollary}[theorem]{Corollary}
\theoremstyle{definition}
\newtheorem{definition}[theorem]{Definition}

\newtheorem{remark}[theorem]{Remark}

\newtheorem*{remark*}{Remark}
\newtheorem*{remarks*}{Remarks}
\newtheorem*{definition*}{Definition}

\usepackage{calrsfs}
\usepackage[T1]{fontenc}
\usepackage{textcomp}
\usepackage{times}
\usepackage[scaled=0.92]{helvet}

\renewcommand{\tilde}{\widetilde}
\newcommand{\isomto}{\overset{\sim}{\rightarrow}}
\newcommand{\R}{\mathbf{R}}
\newcommand{\Z}{\mathbf{Z}}

\newcommand{\PP}{\mathbf{P}}
\newcommand{\T}{\mathbf{T}}
\newcommand{\F}{\mathbf{F}}
\newcommand{\C}{\mathbf{C}}
\newcommand{\PSL}{\mathrm{PSL}}
\newcommand{\PGL}{\mathrm{PGL}}
\newcommand{\Isom}{\mathrm{Isom}}
\newcommand{\na}{\circ}
\newcommand{\f}{\varphi}

\newcommand{\nn}{\nonumber}

\DeclareMathOperator{\Res}{Res}

\DeclareMathOperator{\V}{V}
\DeclareMathOperator{\E}{E}
\DeclareMathOperator{\Aut}{Aut}
\DeclareMathOperator{\dirac}{dirac}
\DeclareMathOperator{\Cay}{Cay}
\DeclareMathOperator{\Coc}{C_{har}}
\DeclareMathOperator{\Sp}{S}

\begin{document}

\date{\today\ (version 2.1)}
\title[Rigidity for Mumford curves]{Measure-theoretic rigidity for Mumford curves}
\author[G.\ Cornelissen]{Gunther Cornelissen}
\author[J.\ Kool]{Janne Kool}
\address{\normalfont{Mathematisch Instituut, Universiteit Utrecht, Postbus 80.010, 3508 TA Utrecht, Nederland}}
\email{\{g.cornelissen,j.kool2\}@uu.nl}
\thanks{We thank Aristides Kontogeorgis for his suggesting the relation to the configuration problem in Remark \ref{ak}. We thank Tom Ward and the referee for their careful reading of the manuscript and useful comments on the exposition.}
\begin{abstract}
\noindent One can describe isomorphism of two compact hyperbolic Riemann surfaces of the same genus  by a measure-theoretic property: a chosen isomorphism of their fundamental groups corresponds to a homeomorphism on the boundary of the Poincar\'e disc that is absolutely continuous for Lebesgue measure if and only if the surfaces are isomorphic. 

In this paper, we find the corresponding statement for Mumford curves, a nonarchimedean analog of Riemann surfaces. In this case, the mere absolute continuity of the boundary map (for Schottky uniformization and the corresponding Patterson--Sullivan measure) only implies isomorphism of the special fibers of the Mumford curves, and the absolute continuity needs to be enhanced by a finite list of conditions on the harmonic measures on the boundary (certain nonarchimedean distributions constructed by  Schneider and Teitelbaum) to guarantee an isomorphism of the Mumford curves. 

The proof combines a generalization of a rigidity theorem for trees due to Coornaert, the existence of a boundary map by a method of Floyd, with a classical theorem of Babbage--Enriques--Petri on equations for the canonical embedding of a curve.  \end{abstract} 

\maketitle

\section*{Introduction}
The celebrated measure theoretic rigidity for hyperbolic Riemann surfaces specializes to the following statement (Mostow \cite{Mostow}, Kuusalo \cite{Kuusalo}): if $X_1$ and $X_2$ are compact hyperbolic Riemann surfaces, written as quotients of the Poincar\'e disk $\Delta$ by Fuchsian groups $\Gamma_1$ and $\Gamma_2$ of the first kind, then $X_1$ and $X_2$ are isomorphic (viz., $\Gamma_1$ and $\Gamma_2$ are conjugate in $\PSL(2,\R)$) if and only if the boundary map $\varphi$ associated to $\Gamma_1$ and $\Gamma_2$ is absolutely continuous with respect to Lebesgue measure. Here, $\varphi$ is a homeomorphism $$\varphi \colon \partial \Delta = S^1 \rightarrow \partial \Delta = S^1$$ that is equivariant w.r.t.\ a chosen \emph{group} isomorphism $\alpha \colon \Gamma_1 \rightarrow \Gamma_2$ in the sense that $$\varphi(\gamma_1 \cdot x) = \alpha(\gamma_1)\varphi(x)$$ for any $x \in S^1 \cong \PP^1(\R)$ (on which $\PSL(2,\R)$ acts naturally by fractional linear transformations), and for any $\gamma_1 \in \Gamma_1$. For higher-dimensional hyperbolic manifolds, the absolute continuity is automatic, giving the celebrated rigidity theorem of Mostow. 

There exist versions of this result for more general Fuchsian groups, i.e., Schottky groups,where, instead of Lebesgue measure, more general Patterson--Sullivan measure on the boundary (=limit set) is used \cite{Patterson} \cite{Sullivan}. There are also analogues for more general subgroups of isometry groups of hyperbolic space that, for example, apply to graphs, cf.\ Bowen \cite{Bowen},  
Coornaert \cite{Coornaert}, Hersonsky and Paulin \cite{HP}, Tukia \cite{Tukia}, Yue \cite{Yue}. 

As is well-known, there is a $p$-adic, or more general nonarchimedean, analogue of uniformization of Riemann surfaces by Schottky groups, namely, Mumford's theory of uniformization by nonarchimedean Schottky groups (cf.\ for example, Mumford \cite{Mumford}, \cite{GvdP}). 

\subsection*{Example}
Let us give an intuitive example of a Mumford curve in positive characteristic $p$ to illustrate all the ingredients. Let $\F_q$ be the finite field with $q=p^n$ elements, let $k=\F_q((T))$ be the valued field of power series in $T$, and let $K$ be a complete algebraically closed extension of $k$. Let $\lambda\in k$ be such that $0<|\lambda|<1$, i.e., let $\lambda=\sum_{i=N}^\infty a_iT^i$, with $a_i\in\F_q$ and $N\geq1$. Then the following equation describes a so called Artin-Schreier-Mumford curve over $K$: 
\[X_\lambda=\{(x,y)|(x^q-x)(y^q-y)=\lambda\}\subset\PP^1(K)\times\PP^1(K).\]
The reduction of this curve, i.e., the curve modulo $T$ over the residue field, which in this example is the finite field $\F_q((T))/(T)\cong\F_q$, is given by 
\[\{(x,y)|(x^q-x)(y^q-y)=0\}\subset\PP^1(\F_q)\times\PP^1(\F_q).\] 
The solution of this equation forms a "chess board" of $q$ horizontal an $q$ vertical copies of $\PP^1(\F_q)$. This object is what is called the special fibre of the curve. The very same Mumford curve can also be obtained by a uniformization process. The uniformizing space is the nonarchimedean analytic space $\PP^1(K)-\Lambda_\Gamma$, where $\Lambda_\Gamma$ is the limit set of a Schottky group $\Gamma\leq\PGL(2,k)$, which acts by M\"obius transformations properly discontinuously on $\PP^1(K)-\PP^1(k)$, having a limit set contained $\PP^1(k)$. A Schottky group corresponding to the curve $X_\lambda$ can be constructed as follows: define 
\[\epsilon_u=\left(\begin{array}{cc}1&u\\0&1\end{array}\right),
\tau=\left(\begin{array}{cc}0&t\\1&0\end{array}\right), \epsilon_u'=\tau\epsilon_u\tau,\]
where $u\in\F_q$ and $t\in k^*$ such that $0<|t|<1$. Let $\Gamma(t)$ be the group generated by the commutators $[\epsilon_u,\epsilon_u']$. Then there is a $t$ such that $X_\lambda=\Gamma(t)\backslash(\PP^1(K)-\Lambda_{\Gamma(t)})$. The relation between $\lambda$ and $t$ is  transcendental and studied in \cite{CoKaKo}. The group $\Gamma(t)$ also acts properly discontinuously on the Bruhat-Tits tree, having limit points in the boundary of the tree. The boundary of the tree is canonically isomorphic to $\PP^1(k)$ and the set of limit points is exactly the same set $\Lambda_\Gamma$ as above. Denote by $\T(\Lambda_\Gamma)$ the tree spanned by the limit set. The quotient $\Gamma\backslash\T(\Lambda_\Gamma)$ is a finite graph which is in fact the intersection graph of the special fibre. In this example this is the regular bi-partite graph with $2q$ vertices; a vertex for each copy of $\PP^1(\F_q)$. \\

In exact terms, Mumford's uniformization theory by nonarchimedean Schottky groups comprises the following. A projective curve $X$ of genus $g \geq 2$ over a complete algebraically closed nonarchimedean valued field $K$ can be considered as a one-dimensional compact rigid analytic space (in the sense of Tate \cite{Tate}), and this analytic space is a quotient $\Gamma \backslash (\PP^1(K)-\Lambda_\Gamma)$ for some discrete subgroup $\Gamma \subseteq \PGL(2,k)$ with limit set $\Lambda_\Gamma \subseteq \PP^1(k)$(cf. Gerritzen and Van der Put \cite{GvdP} prop\ 1.6.4)) precisely if the special fibre of a stable model $X$ over the ring of integers of $k$ is totally split (i.e., a union of rational curves over the algebraic closure of the residue field). Here $k\subset K$ is a discrete valued complete nonarchimedean subfield. 

To emphasize the analogy, let us briefly recall classical Schottky uniformization.  A Riemann surface $X$ can be described as a quotient $\Gamma\backslash (\PP^1(\C)-\Lambda_\Gamma)$, where $\Gamma\subseteq\PGL(2,\C)=\Aut(\PP^1(\C))$ is a free discrete group of rank $g$ (=genus of $X$) with limit set $\Lambda_\Gamma\subset \PP^1(\C)$. Manin and others have argued that the classical complex analogue of the special fibre should be the set of closed geodesics in the solid handlebody $\Gamma \backslash \mathbb{H}^3$ given by the quotient of real hyperbolic 3-space $\mathbb{H}^3$ by $\Gamma \hookrightarrow \PSL(2,\C) = \Aut(\mathbb{H}^3)$ (compare \cite{Manin}, \cite{MaMa}). Table \ref{table1} gives a tabular representation of the analogy. 

\begin{table}[h]

\begin{tabular}{lll}
& Complex case & Nonarchimedean case \\ \hline
Ground field & $\C$ & $k$ complete discrete nonarchimedean  \\
& & $k \subseteq K$ complete algebraically closed  \\  
Free group of rank $g$ & $F_g$ & $F_g$ \\ 
Schottky group & $F_g \cong \Gamma \hookrightarrow \PGL(2,\C)$ & $F_g \cong \Gamma \hookrightarrow \PGL(2,k)$ \\
Limit set & $\Lambda_\Gamma \subseteq \PP^1(\C)$ &  $\Lambda_\Gamma \subseteq \PP^1(k) \subseteq \PP^1(K)$ \\
Curve & $X(\C) \cong \Gamma \backslash ( \PP^1(\C)-\Lambda_\Gamma)$ & $X(K) \cong \Gamma \backslash ( \PP^1(K)-\Lambda_\Gamma)$ \\
Hyperbolic space & $\mathbb{H}^3$ & $\T$ (Bruhat-Tits tree) \\
Boundary &    $\partial \mathbb{H}^3 = \PP^1(\C)$   & $\partial \T = \PP^1(k)$ \\ 
Special fiber &  $\{$closed geodesics in $\Gamma \backslash \mathbb{H}^3 \}$      &   $   \mathcal{X} \otimes \bar k$  \\
& & ($\mathcal{X}=$ model of $X$ over integers of $k$) \\
Subspace & $\mathbb{H}^3(\Lambda_\Gamma)=$   & $\T(\Lambda_\Gamma)=$ \\ 
 & $\{$geodesics ending on $\Lambda_\Gamma\}$ & subtree spanned by limit set\\
Dual of special fiber & $\Gamma \backslash \mathbb{H}^3(\Lambda_\Gamma)$ & $\Gamma \backslash \T(\Lambda_\Gamma)$\\
 \hline \hline
\end{tabular}
\caption{Analogy between classical and nonarchimedian Schottky uniformization}
\label{table1}
\end{table}
In this paper, we consider the analog of measure-theoretic rigidity for the case of Mumford curves: can one describe isomorphism of Mumford curves by measure-theoretic properties of an associated boundary map?

Let us first explain why an obvious analog of the complex analytic statement is false.
 A nonarchimedean Schottky group $\Gamma \subseteq \PGL(2,k)$ is also a group of automorphisms of the tree $T(\Lambda_\Gamma)$ spanned by the limit set as a subtree of the  Bruhat-Tits tree over the residue field of $k$. This limit set indeed carries a  Patterson-Sullivan measure of dimension $e(\Gamma)$ (see section \ref{trees}), but absolute continuity of a boundary map for two such groups will only imply conjugacy of Schottky groups in the isometry group of the tree, \emph{not} inside $\PGL(2,k)$. Now $\Aut(T(\Lambda_\Gamma))$ is much larger than its ``linear cousin'' $\PGL(2,k)$, and conjugacy of Schottky groups inside $\Aut(T(\Lambda_\Gamma))$ is, as we will see, only equivalent to an isomorphism of the \emph{special fibers} of the two curves. There are many non-isomorphic curves with isomorphic special fibers---actually, for any given curve, uncountably many. Another manifestation of this fact is the statement that the group $\Gamma$ has uncountably many conjugacy classes inside $\PGL(2,k)$, but only countably many inside $\Isom(T(\Lambda_\Gamma))$ (this observation was communicated to us by Lubotzky, compare \cite{Lubotzky}). 

To remedy this problem, we suggest to consider, in addition to Patterson-Sullivan measure, the so-called \emph{harmonic measures of weight $\ell$} on the boundary, in the sense of Schneider and Teitelbaum (\cite{Schneider}, \cite{Teitelbaum}). For $\ell=2$ we call these simply harmonic measures, and denote the $g$-dimensional vector space of such by $\Coc(\Gamma,2)$. For $\ell\geq 3$, the harmonic measures of weight $\ell$ for $\Gamma$ span a $(\ell-1)(g-1)$-dimensional vector space $\Coc(\Gamma,\ell)$. Denote by $\Omega_{X_\Gamma}$ the canonical bundle over the curve $X_\Gamma$. Teitelbaum's Poisson kernel theorem gives, for every even $\ell$, explicit back and forth isomorphisms $(\mathrm{Poisson}(\ell),\mathrm{Res}(\ell))$ between the space of weight-$\ell$ harmonic measures and the space of holomorphic $\ell/2$-differential forms on the corresponding Mumford curve $X_\Gamma$ (a.k.a.\ the space of weight $\ell$ modular forms for $\Gamma$): 
$$ \xymatrix{ \Coc(\Gamma,\ell)  \ar@/^1.5pc/[r]^{\mathrm{Poisson(\ell)}}   & H^0(X_\Gamma,\Omega_{X_\Gamma}^{\otimes \ell/2})   \ar@/^1.5pc/[l]^{\mathrm{Res}(\ell)}  } $$

Thus, we can define a \emph{product} of harmonic measures: let $\boldsymbol{\ell}=(\ell_1,\dots,\ell_N)$ be a vector of strictly positive even integers, and define the product map $m_{\boldsymbol{\ell}}$ by the commutativity of the following diagram
\begin{equation} \label{dia}  \xymatrix{   \bigotimes\limits_{i=1}^N \Coc(\Gamma,\ell_i)   \ar@{->}[r]_{m_{\boldsymbol{\ell}}} \ar@{->}[d]_{\bigotimes \mathrm{Poisson}(\ell_i)} &  \Coc(\Gamma,\sum\limits_{i=1}^N \ell_i) \\                    
 \bigotimes\limits_{i=1}^N H^0(X_\Gamma,\Omega_{X_\Gamma}^{\otimes \ell_i/2})  \ar@{->}[r]  &  H^0(X_\Gamma,\Omega_{X_\Gamma}^{\otimes \sum\ell_i/2}) \ar@{->}[u]_{\mathrm{Res}(\sum \ell_i)}  } \end{equation} 
in which the bottom arrow is just the usual product of differential forms. 
We use the shorthand notation $R_2(\Gamma):=\ker(m_{(2,2)})$, $R_3(\Gamma):=\ker(m_{(2,2,2)})$ and $R_4:=\ker(m_{(2,2,2,2)})$. These kernels describe the linear relations between products of degree two, three and four of harmonic measures.  

Now our result is the following: 

\begin{introtheorem} \label{mainthm} Let  $k$ be a discrete valued nonarchimedean field, and let $K$ denote a complete and algebraically closed field extension of $k$. Let $\Gamma_1,\Gamma_2\leq\PGL(2,k)$ be Schottky groups of the same genus $g\geq 2$, defining Mumford curves $X_1$ and $X_2$, respectively over $K$. Let $\varphi$ denote the boundary map associated to an isomorphism of $\Gamma_1$ with $\Gamma_2$.
 
\begin{itemize}
\item[\textup{(i)}] 
The special fibers of $X_1$ and $X_2$ are isomorphic if and only if there is an isometry \[F:T(\Lambda_{\Gamma_1})\to T(\Lambda_{\Gamma_2})\]  which conjugates the Schottky groups $\Gamma_i$, if and only if the corresponding Patterson-Sullivan measures are of the same conformal dimension and the map $\varphi$ is absolutely continuous for these measures. \\
If this is the case, then pullback by $F^{-1}$ induces an isomorphism of the spaces of harmonic measures 
$$ F_* \colon \Coc(\Gamma_1,2) \isomto \Coc(\Gamma_2,2).$$
\item[\textup{(ii)}] Suppose that the curves $X_i$ are not hyperelliptic. Then the curves $X_1$ and $X_2$ are isomorphic over $K$ if and only if the corresponding Schottky groups $\Gamma_i$ are conjugate in $\PGL(2,k)$, if and only if the map $\varphi$ is absolutely continuous with respect to Patterson--Sullivan measure \emph{and} $F_*$ respects the linear relations of degree $\leq 4$ between the harmonic boundary measures in the sense that $$F_*R_2(\Gamma_1) \subseteq R_2(\Gamma_2), F_*R_3(\Gamma_1)\subseteq R_3(\Gamma_2) \mbox{ and }F_*R_4(\Gamma_1)\subseteq R_4(\Gamma_2).$$ The relations in $R_4$ are only necessary if $g=3$. The relations in $R_3$ are only necessary if $X$ is trigonal (i.e., admits a map $X \rightarrow \PP^1$ of degree 3), or if $X$ is isomorphic to a plane quintic. 
\end{itemize}
\end{introtheorem}

The theorem says that when the usual measure-theoretic property of absolute continuity with respect to Patterson-Sullivan measure is enhanced to include preservation of linear relations in three finite-dimensional vector spaces of harmonic measures, then rigidity of Mumford curves follows. Note that Mumford proved in (\cite{Mumford}, Corollary 4.11) that two Mumford curves with Schottky groups in $\PGL(2,k)$ are isomorphic over $K$ if and only if the corresponding Schottky groups are conjugate over $k$. 

We do not know at present how to deal with the hyperelliptic case. 

The proof of the theorem uses the theorem of Petri et al.\ on the equations for the image of the canonical 
embedding of a non-hyperelliptic curve. The Poisson kernel theorem of Teitelbaum is used to translate this algebro-geometric statement back into a statement about harmonic measures. 

The contents of the paper is as follows: in the first part, we prove a version of Coornaert's ergodic rigidity theorem for two different trees (the original theorem was for one tree), and apply it to the intersection dual graph of the special fibre of two Mumford curves---this will prove the first part of the theorem. Then we give the proof of the second part of the theorem.

\section{Measure-theoretic rigidity for trees and special fibre isomorphism}\label{trees}
Throughout, we only consider locally finite trees $T$ without ends, that is, trees such that for any vertex the valency is finite and the complement is disconnected. We equip the set of vertices of $T$ with a metric $d$ in which all edges are of length one. All graphs will be oriented. For a graph $G$, we denote the set of vertices with $\V=\V(G)$ and the set of edges with $\E=\E(G)$. We denote the set of edges from $x \in \V(G)$ to $y \in \V(G)$ with $\E(x,y)$. A \emph{basic subset} of a tree is a set containing all vertices and edges which are on the target side of a given edge, i.e., an edge $e\in\E(x,y)$ defines the basic set $A=A(e)$ where $v\in A$ if $d(v,y)<d(v,x)$. 

\begin{figure}[ht]
\begin{center}
\includegraphics[width=0.7\textwidth]{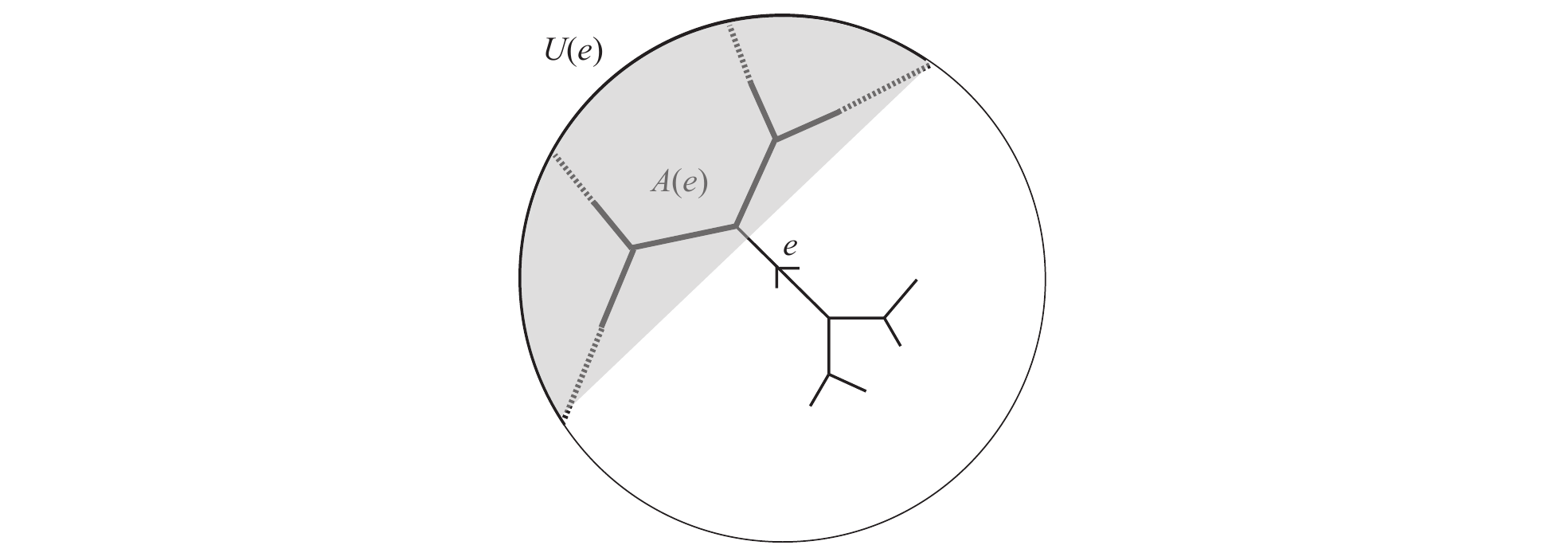}
\end{center}
\caption{Basic set $A(e)$ and compact open set $U(e)$}
\label{AeUe}
\end{figure}

\subsection*{Automorphism groups} Let $G=\Aut(T)$ be the automorphism group of a tree $T$. Equip $G$ with the topology in which the stabilizers of the vertices form the compact open neighborhoods of the identity element. Denote by $G^+$ the subgroup of $G$ of automorphisms which act without inversion, i.e., automorphisms which do not switch the orientation of any edge. The group $G^+$ can be partitioned into  so-called \emph{elliptic automorphisms} (which have a fixed vertex), and \emph{hyperbolic automorphisms} (which do not fix any vertex). 

We are concerned with certain discrete subgroups of $G^+$ called tree-Schottky subgroups. 

\begin{definition}
A subgroup $\Gamma$ of $G^+$ is called a \emph{tree-Schottky group} if it is finitely generated, and all non-identity elements are hyperbolic.
\end{definition}

In particular, a tree-Schottky group is a discrete and free group (see  Serre \cite{Serre} \textsection I.3.3). The rank of the group is called the \emph{genus} of the group. 
\begin{remark}\label{funddomain} Tree-Schottky groups satisfy remarkable geometric properties, which can be used to give the following second equivalent definition (for a proof of the equivalence see Lubotzky \cite{Lubotzky}, section 1).
A tree-Schottky group $\Gamma$ of rank $g$ is a free subgroup of $G^+$ such that for a choice of generators $\Gamma=\left\langle \gamma_1,...\gamma_g\right\rangle$ there exist distinct basic sets $A_i,B_i\subset T$ for $1\leq i\leq g$ satisfying:
\begin{align*}
\gamma_i(A_i) &=T-B_i\nn\\
\gamma_i^{-1}(B_i)&=T-A_i.
\end{align*}
\end{remark}

\subsection*{Patterson--Sullivan measures}
The \emph{boundary} $\partial T$ of the tree $T$ is defined as the space of geodesic rays $r:\Z_{\geq 0}\to \E(T)$ modulo the following equivalence relation: two rays $r, r'$ are considered equivalent if there are numbers $t_0, s$, such that for all $t>t_0$, $r(t)=r'(t+s)$.  The metric on  $\partial T$ with respect to a fixed base point $x_0\in T$ is defined as follows. Every boundary point $\xi\in\partial T$, has a unique representative ray $r:\Z_{\geq 0}\to E(T) $ with initial point of $r(0)$ equal to $x_0$. The image of $r$ is denoted by $[x_0,\xi]$.

\begin{figure}[ht]
\begin{center}
\includegraphics[width=0.9\textwidth]{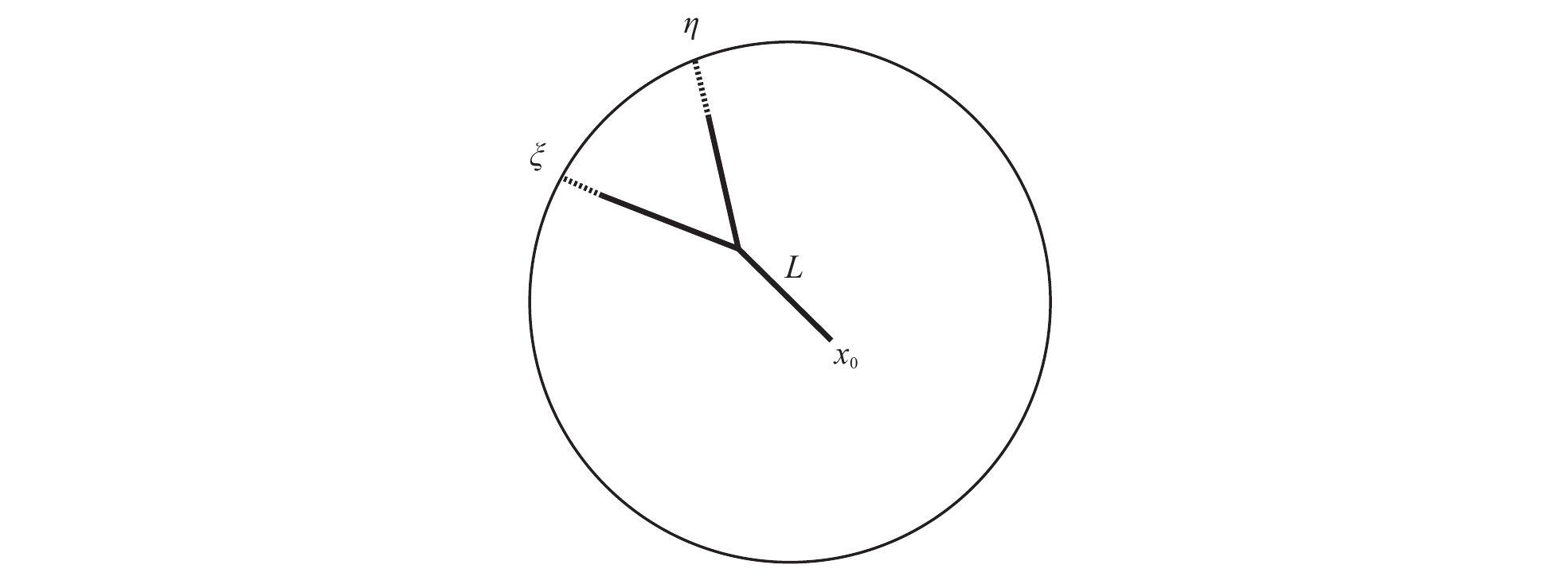}
\end{center}
\caption{Distance between boundary points}
\end{figure}

The \emph{distance between two boundary points} $\eta,\xi\in\partial T$ is given by
\[\rho_{x_0}(\eta,\xi)=e^{-L},\]
where $L$ is the length of $[x_0,\eta]\cap[x_0,\xi]$. We fix $x_0$ from now on and leave it out of the notation. 

This metric induces a topology on the boundary. The set of all boundary points that have a representative ray that is entirely contained in a basic set $A(e)$ is denoted by $U(e)$. One can check that $U(e)$ are compact open subsets of the boundary.  

An isometry $\gamma$ of $T$ induces a homeomorphism of $\partial T$ to itself.  The map measuring locally the dilatation by $\gamma$ is $$j_\gamma(\xi)=e^{d(x_0,p)-d(p,\gamma^{-1} x_0)},$$ where $p$ is the projection of $\xi$ on $[x_0,\gamma^{-1} x_0]$.

Let $\Gamma$ be a group of isometries acting properly discontinuously on $T$. A Borel measure $\mu$ on $\partial T$ is called $\Gamma$-conformal of dimension $d$, if for all $\gamma\in\Gamma$ $\gamma^*\mu=j^d_\gamma\mu$, where $\gamma^*\mu$ is defined by $\gamma^*\mu(A)=\mu(\gamma(A))$. In \cite{CoornaertPacific}, Coornaert constructs such a measure  completely analogously to the Patterson--Sullivan measures on the limit set of a Fuchsian group. That is
\[\mu(A)=\lim_{s\to e(\Gamma)}\frac{\sum_{y\in\Gamma(x)}e^{-sd(y,x_0)}\dirac(y)(A)}{\sum_{y\in\Gamma(x)}e^{-sd(y,x_0)}},\]
here $e(\Gamma)$ is the critical exponent for the Poincar\'e series (which is the denominator in the above formula). In more concrete terms, denote with $n_Y(R)$ the number of points in a orbit $Y$ within a distance $R$ from the basepoint $x_0$. Then 
\[e(\Gamma)=\limsup_{R\to\infty}\frac{1}{R}\log n_Y(R).\]
The critical exponent does neither depend on the orbit $Y$, nor on the basepoint $x_0$, and the measure $\mu$ is $\Gamma$-conformal of dimension $e(\Gamma)$. This measure however, does depend on the base point $x_0$, but measures with different base points are absolutely continuous with respect to each other. For a $\Gamma$-conformal measure $\mu$ of dimension $d$ on  $\partial T$, a $\Gamma$-invariant measure $\nu$ on $$\partial^2T=\{(\xi,\eta)|\xi,\eta\in\partial T, \xi\not=\eta\}$$ is given by the following formula:
\[
\nu=\frac{\mu\times\mu}{\rho(\xi,\eta)^{2d}}.
\]
In particular, $\nu$ is supported on the limit set $\Lambda_\Gamma\times\Lambda_\Gamma$, and is ergodic with respect to the group action on the limit set, (\cite{CoornaertPacific}, Theorem 7.7). In contrast to $\mu$, the measure $\nu$ does not depend on the base point.
 
\subsection*{Rigidity theorem for trees} These measures are used in the following rigidity theorem, which is a slight adaptation of the main theorem of Coornaert \cite{Coornaert} to the setting of two (a priori different) trees. 

\begin{theorem}\label{treerigidity}
For $i=1,2$, let $T_i$ be a complete locally compact tree without endpoints, and let $\Gamma_i$ be an isometry group acting properly discontinuously on $T_i$. Let $\mu_1$ and $\mu_2$ be $\Gamma_1$- resp. $\Gamma_2$-conformal measures on $\partial T_1$ resp. $\partial T_2$, of the same dimension $d\geq 0$, depending on reference vertices $x_0$, $y_0$ in $T_1$ an $T_2$ respectively. If there is a measurable bijection $\f:\partial T_1\to \partial T_2$ and a bijection $\alpha:\Gamma_1\to\Gamma_2$ satisfying the conditions:

\begin{description}
\item [equivariance] for $\mu_1$ almost all $\xi\in\partial T_1$ and for all $\gamma_1\in\Gamma_1$, $\f(\gamma_1\xi)=\alpha(\gamma_1)\f(\xi)$,
\item[non-singular] for all Borel sets $A\subset T_1$, $\mu_1(A)=0$ if and only if $\mu_2(\f(A))=0$, 
\item[ergodic] \label{ergodic} the action of $\Gamma_1$ on $(\partial^2 T_1,\nu_1)$ is ergodic, where $\nu_1=\frac{\mu_1\times\mu_1}{\rho_{x_0}(\xi,\eta)^{2d}}$,
\item[support] the support of $\mu_1$ is $\partial T_1$, \label{support}
\end{description}
then there exists a bijective isometry $F:T_1\to T_2$, such that $\f(\xi)=F(\xi)$ for $\mu_1$-almost all $\xi$ in $\partial T_1$. Moreover, $\alpha$ is a group isomorphism with $\alpha(\gamma)=F\na\gamma\na F^{-1}$ for all $\gamma$ in $\Gamma_1$.
\end{theorem}

\begin{proof}
The main ingredient of the proof is the invariance of the \emph{cross-ratio} under $\f$. Recall that the cross-ratio $c(\xi_1,\xi_2,\xi_3,\xi_4)$ of four pairwise disjoint elements of $\partial T_1$ is the quotient
\begin{equation} \label{crossratio} c(\xi_1,\xi_2,\xi_3,\xi_4)=\frac{\rho(\xi_3,\xi_1) \cdot \rho(\xi_4,\xi_2)}{\rho(\xi_3,\xi_2) \cdot \rho(\xi_4,\xi_1)}.\end{equation}
This ratio is independent of $x_0$, and is in fact the same as 
\begin{equation} \label{crossequal}
c(\xi_1,\xi_2,\xi_3,\xi_4)=e^{-L},
\end{equation} 
where $L$ is the signed length of the intersection $[\xi_1,\xi_2]\cap[\xi_3,\xi_4]$, where the sign depends on the orientation of the geodesics.
\begin{figure}[ht]
\begin{center}
\includegraphics[width=0.4\textwidth]{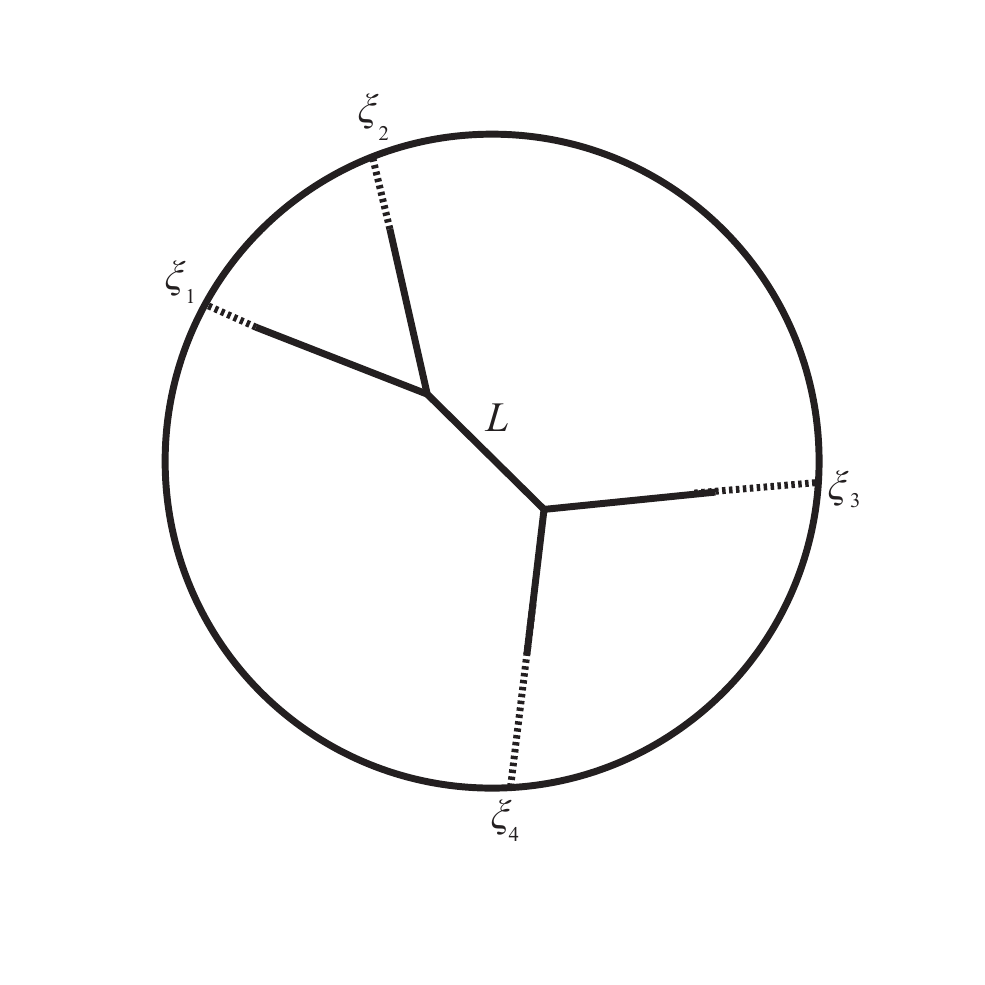}
\end{center}
\caption{Cross-ratio $c(\xi_1,\xi_2,\xi_3,\xi_4)=e^{-L}$}
\end{figure}

Note that $\Gamma_2$ acts ergodically on $(\partial^2 T_2,\nu_2)$; let $B\subset \partial^2T_2 $ be  $\Gamma_2$-invariant, then $(\f\times\f)^{-1}(B)$ is $\Gamma_1$-invariant, and hence, since $\nu_1$ is \textbf{ergodic}; 
\[\nu_1((\f\times\f)^{-1}(B))=0 \text{ or } \nu_1(\partial^2T_1-(\f\times\f)^{-1}(B))=0.\]
So in particular,
\[\mu_1\times\mu_1((\f\times\f)^{-1}(B))=0 \text{ or } \mu_1\times\mu_1(\partial^2T_1-(\f\times\f)^{-1}(B))=0.\]
Using (\textbf{non-singular}), we get 
\[\mu_2\times\mu_2(B)=0 \text{ or } \mu_2\times\mu_2(\partial^2T_2-B)=0,\]
and hence 
\[\nu_2(B)=0 \text{ or } \nu_2(\partial^2T_2-B)=0.\]
   
By (\textbf{non-singular}) the measure $(\f\times\f)_*\nu_1$ on $\partial^2 T_2$ is absolutely continuous with respect to $\nu_2$. Therefore, there exists a measurable function $h$, the Radon-Nikodym derivative, such that $(\f\times\f)_*\nu_1=h \nu_2$. The function $h$ is $\Gamma_2$-invariant, and hence by the ergodicity of $\nu_2$, the function $h$ is almost everywhere constant with respect to $\nu_2$. Likewise, let $g$ be a function on $\partial T_1$ such that $\f^{-1}_*\mu_2=g \mu_1$. Now let $A$ be any Borel set in $\partial^2 T_1$, and let $B=(\f\times \f)(A)\subset\partial^2 T_2$, then
\begin{eqnarray*} \frac{\mu_1\times\mu_1}{\rho_{x_0}(\xi,\eta)^{2d}}(A)&=&\nu_1(A)\\&=&(\f\times \f)_*\nu_1(B)\\&=&h\cdot\nu_2(B)\\&=&h\cdot\frac{\mu_2\times\mu_2}{\rho_{y_0}(\f(\xi),\f(\eta))^{2d}}(B)\\&=&g^2h\cdot\frac{\mu_1\times\mu_1}{\rho_{y_0}(\f(\xi),\f(\eta))^{2d}}(A). \end{eqnarray*}
We find a function $f=(g^2h)^{1/2d}$ such that the following equation holds $\nu_1$ almost everywhere:
\[\rho_{y_0}(\f(\xi_1),\f(\xi_2))=f(\xi_1)f(\xi_2)\rho_{x_0}(\xi_1,\xi_2).\]
It follows that the cross-ratio is invariant under $\f$ almost everywhere, i.e., there is a subset $W\subset\partial T_1$ of full $\mu_1$ measure on which $\f$ is well defined and preserves the cross-ratio. 
The tree $T_1$ is without ends and the support of $\mu_1$ is $\partial T_1$. We will first define the map $F$ for vertices $v\in\V(T_1)$ for which there exists a triple $\xi_1,\xi_2,\xi_3$ of pairwise disjoint elements of $W$ such that $v$ is the center of this triple, i.e., 
\[v=[\xi_1,\xi_2]\cap[\xi_1,\xi_3]\cap[\xi_2,\xi_3].\]
Define $F(v)$ to be:
\[F(v)=[\f(\xi_1),\f(\xi_2)]\cap[\f(\xi_1),\f(\xi_3)]\cap[\f(\xi_2),\f(\xi_3)].\]
Extend the map $F$ uniquely to a map $T_1\to T_2$ by linear interpolation along the rays.
It is  necessary  to prove that $F$ is well defined. Let $\xi_4\in W$ be such that $v$ is also the center of $\xi_1,\xi_2,\xi_4$. Then the projection of $\xi_4$ on 
\[[\xi_1,\xi_2]\cup[\xi_1,\xi_3]\cup[\xi_2,\xi_3]\]
lies on $[v,\xi_3]$. The invariance of the cross ratio shows that the projection of $\f(\xi_4)$ lies on $[F(v),\f(\xi_3)]$. 
\begin{figure}[ht]
\begin{center}
\includegraphics[width=0.9\textwidth]{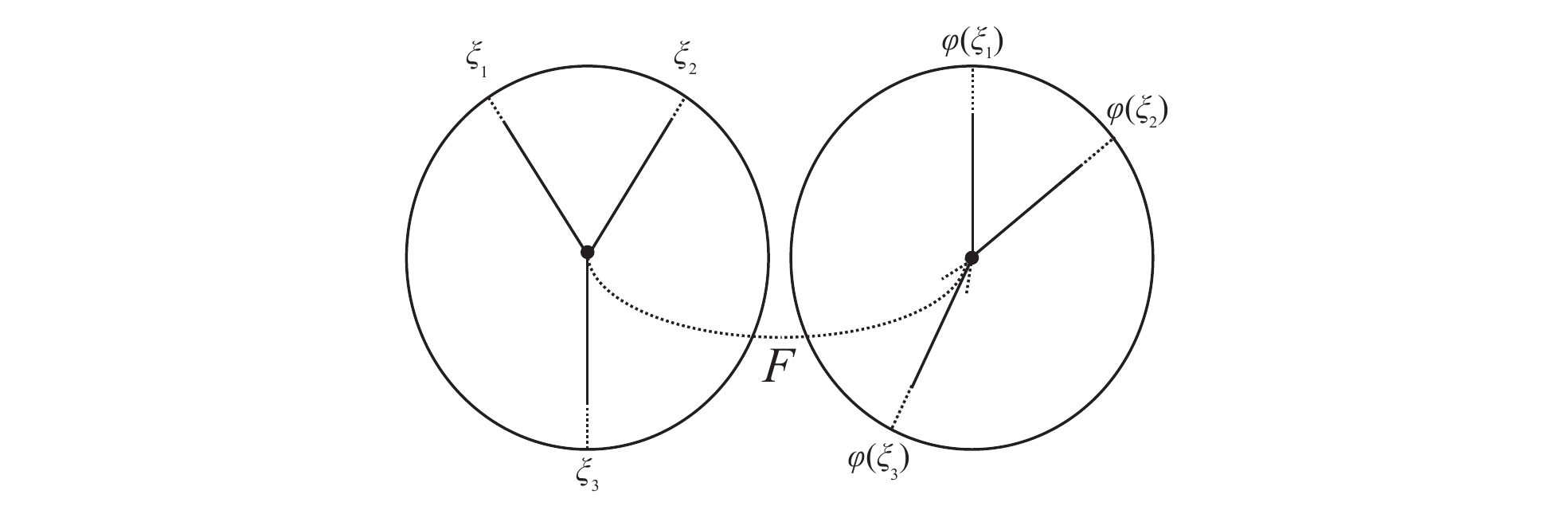}
\end{center}
\caption{Definition of isometry $F$}
\label{isometryF}
\end{figure}

The map $F$ is an isometry; given $v,w\in\V(T_1)$ there are two pair of points in $W$, say, $\xi_1,\xi_2$ and $\xi_3,\xi_4$ such that $[v,w]=[\xi_1,\xi_2]\cap[\xi_3,\xi_4],$ and then the distance between $v$ and $w$ is given by 
\begin{eqnarray*} d(v,w)&=&\ln(c(\xi_1,\xi_2,\xi_3,\xi_4))\\&=&\ln(c(\varphi(\xi_1),\varphi(\xi_2),\varphi(\xi_3),\varphi(\xi_4)))\\&=&d(F(v),F(w)).\end{eqnarray*}
 
For  $\xi\in W$, and hence $\mu_1$ almost everywhere it holds that $F(\xi)=\f(\xi)$. Moreover, the isometry of $T_2$ given by $F\na\gamma_1\na F^{-1}$ coincides with $\alpha(\gamma_1)$, for all $\gamma_1\in\Gamma_1$.
\end{proof} 

\subsection*{Existence of boundary morphism in the tree-Schottky case}
In this section, we prove that for two tree-Schottky groups of the same genus $g$ there is always an equivariant homeomorphism between the limit sets (i.e., maps $\varphi$ and $\alpha$ satisfying (1) as in Theorem \ref{treerigidity}). This is a consequence of the existence of a canonical identification between the limit set of a Schottky group and the Floyd boundary of the free group $F_g$, the latter being independent of the realisation of $F_g$ as Schottky group. 

First, recall the following definitions: A metric space $(X,d_X)$ is called hyperbolic if there exists a $\delta\geq0$ such that for any base point $x_0$ in $X$, $$(x\cdot y)\geq\min((x\cdot y)(y\cdot z))-\delta \text{ for any } x,y,z\in X.$$ Here, $(x\cdot y)$ denotes the Gromov product with respect to $x_0$ defined as $$(x\cdot y)=(d_X(x,x_0)+d_X(y,x_0)-d_X(x,y))/2.$$ This implies in particular that trees are hyperbolic spaces (with $\delta=0$). Secondly, a finitely generated group is called hyperbolic if its Cayley graph is hyperbolic. In particular, free groups are hyperbolic. The \emph{boundary $\partial G$ of a group $G$} is defined by Floyd in \cite{Floyd}: it is the complement of the Cayley graph of $G$ inside its completion, by considering it as a metric space for the ``Floyd metric'' that assigns to two adjacent words $w_1$ and $w_2$, (i.e., with $|w_1^{-1}w_2|=1$), the length $$1/\max\{|w_1|,|w_2|\}^2$$ (with $|\cdot|$ the usual word length, and the empty word having length zero).  Then we have: 

\begin{theorem}[\cite{Coornaertboek}, Theorem 4.1]\label{boundarymorphism}
Let $X$ be a proper geodesic space and let $\Gamma$ be an isometry group of $X$ which acts properly and discontinuously on $X$, and let the quotient of this action be compact. Then $\Gamma$ is hyperbolic if and only if $X$ is hyperbolic. Moreover, if $\Gamma$ is hyperbolic there is a canonical homeomorphism $\partial\Gamma\to\partial X$.
\end{theorem}

\begin{remark} In our case, let $F_g$ denote a free group of rank $g$, let $A=\{A_1,...,A_g\}$ denote a set of generators, and let \[\rho: F_g\to \Aut^+(T)\]
be a faithful representation, such that the image $\rho(F_g)=\Gamma$ is a tree-Schottky group for the tree $T$. Denote by $\Cay(F_g, A)$ the Cayley graph with respect to the given generating set $A$, pick a base point $x_0$ and define
\[\varphi: \Cay(F_g,A)\to T,w\mapsto \rho(w)(x_0).\]
The boundary morphism is obtained by extending this map to the completions on both sides, and then restricting to the respective boundaries, giving a map 
$$ \varphi \colon \partial F_g \isomto \partial T. $$ 
\end{remark}

The above theorem indeed implies our claim: 

\begin{corollary}\label{phiexists}
For two tree-Schottky groups $\Gamma_1, \Gamma_2$ of the same rank $g \geq 2$, and acting on two trees $T_1, T_2$, there is a  group isomorphism $\alpha \colon \Gamma_1 \rightarrow \Gamma_2$ and a homeomorphism $\f:\Lambda_{\Gamma_1}\to\Lambda_{\Gamma_2}$ between the respective limit sets, that is equivariant w.r.t.\ $\alpha$. \end{corollary}

\begin{proof}
Let $T_{\Gamma_i}$ be the subtree of the $T_i$ containing all geodesics connecting the elements of the limit set $\Lambda_{\Gamma_i}$. Let $\rho_i \colon F_g \hookrightarrow \Aut^+(T_i)$ denote the representation of the free group $F_g$ on $g$ letters; then $\alpha:=\rho_2 \circ \rho_1^{-1}$ is a homomorphism between $\Gamma_1$ and $\Gamma_2$.

  The quotient graph $\Gamma_i\backslash T_{\Gamma_i}$ is finite, and hence compact. Hence all requirements of Theorem \ref{boundarymorphism} are fulfilled, and therefore we obtain two homeomorphisms $\f_i:\partial F_g\to\Lambda_{\Gamma_i}$. One can then set $\f=\f_2\na\f_1^{-1}$, and this is equivariant w.r.t.\ $\alpha$ by construction. 
  $$ \xymatrix{    & & \Lambda_{\Gamma_1} \ar[dd]^{\varphi} \\ \partial {F}_g
\ar[urr]^{{\varphi_1}} \ar[drr]^{{\varphi_2}} & \\ & &
\Lambda_{\Gamma_2}}$$
\end{proof}

\subsection{Rigidity theorem for tree-Schottky groups}

We end this section with a theorem on the rigidity of tree-Schottky groups.
 
\begin{theorem}\label{treerigidity2}
Let $T_i$ be trees $(i=1,2)$ and let $\Gamma_i$ be corresponding tree-Schottky groups, both of genus $g$, such that $\Lambda_{\Gamma_i}=\partial T_i$. Then the following are equivalent:
\begin{enumerate}
\item\label{een} The quotient graphs $\Gamma_i\backslash T_i$ are isomorphic.
\item\label{twee} There is an isomorphism $\alpha:\Gamma_1\to\Gamma_2$ and a bijective isometry $F:T_1\to T_2$ such that for all $\gamma\in\Gamma_1$,
\[F\na\gamma=\alpha(\gamma)\na F\]
\item\label{drie} There is an isomorphism $\alpha:\Gamma_1\to\Gamma_2$ and an equivariant homeomorphism $\f:\Lambda_{\Gamma_1}\to\Lambda_{\Gamma_2}$ which is absolutely continuous with respect to the respective Patterson--Sullivan measures. Moreover, these measures are of the same dimension.
\end{enumerate}
\end{theorem}

\begin{proof}
Statement (\ref{drie}) implies (\ref{twee}): this is a direct consequence of Theorem \ref{treerigidity}. We check that all conditions of Theorem \ref{treerigidity} are satisfied. The homeomorphism $\f$ is equivariant and absolutely continuous by assumption. The Patterson--Sullivan measures are of the same dimension an the Patterson--Sullivan measure of the first group is supported on $\Lambda_{\Gamma_1}=\partial T_1$. As noted above, the group $\Gamma_1$ acts ergodically on $\Lambda_{\Gamma_1}^2=\partial^2T_1$.

Statement (\ref{twee}) implies (\ref{een}): suppose there is an isometry as in (\ref{twee}), and let $y_1\in\V(T_1)$. Then for any $y\in\Gamma y_1$ it holds that $$F(y)=F(\gamma(y_1))=\alpha(\gamma)F(y)\in\Gamma_2 F(y).$$ Therefore, $F$ sends orbits of vertices to orbits of vertices. Moreover, because $F$ is an isometry, it also sends orbits of edges to orbits of edges. Hence, $F$ induces an isomorphism between the quotient graphs. 

What is left to prove is that (\ref{een}) implies (\ref{drie}). For this, denote $G_i=\Gamma_i\backslash T_i$, and let $f:G_1\to G_2$ be the given isomorphism. 

A graph $G$ is called \emph{combinatorial} if $\E(x,y)$ contains at most one edge and if $\E(x,x)=\emptyset$ for all $x\in\V(G)$. If $G_1$ is not combinatorial we make it combinatorial by replacing every edge by three edges while adding two new vertices, i.e.,

\[\xymatrixcolsep{4cm}\xymatrix{\bullet\ar@{-}[r]&\bullet}\mbox{ is replaced by }\xymatrixcolsep{1cm}\xymatrix{\bullet\ar@{-}[r]&\bullet\ar@{-}[r]&\bullet\ar@{-}[r]&\bullet}\]

We do this simultaneously with $T_1$ and denote the new graphs with $\bar G_1$, and $\bar T_1$. There is a unique way of extending the action of $\Gamma_1$ on $T_1$ to $\bar T_1$, and the quotient $\Gamma_1\backslash \bar T_1$ is $\bar G_1$. Repeat the procedure with $G_2$, such that $f$ extends to an isomorphism $\bar f:\bar G_1\to\bar G_2$. From now on we will assume that the quotient graphs are already combinatorial. Let $P_1\subset G_1$ be a maximal subtree with $2g$ endpoints. The tree $T_1$ is the universal covering of $G_1$, and hence there is a lift $\ell_1:P_1\hookrightarrow T_1$, which is unique once one point of the lift is chosen. Moreover, $\ell(P_1)$ is a fundamental domain for $\Gamma_1$, and therefore there are $g$ group elements $\gamma_1,...,\gamma_g$ of $\Gamma_1$ and there is a labeling $a_1,b_1,\dots,a_g,b_g$ of the end points of $P_1$ such that for $1\leq i\leq g$, \[d(\gamma_i(\ell_1(a_i)),\ell_1(b_i))=1.\] The group $\Gamma_1$ is generated by $\gamma_1,\dots,\gamma_g$. Let $P_2=f(P_1)$, which is a maximal subtree of $G_2$. Again, after the choice of one point there is a unique lift $\ell_2:P_2\hookrightarrow T_2$. Moreover, there are $g$ elements $\delta_1,\dots,\delta_g$ generating $\Gamma_2$ and satisfying 
\[d(\delta_i(\ell_2\na f(a_i)),\ell_2\na f(b_i))=1.\]
There are representations 

\[\rho_1:F_g\to \Isom(T_1) \text{ and }  \rho_2:F_g\to \Isom(T_2),\] such that 

\[\rho_2\circ\rho_1^{-1}(\gamma_i)=\alpha(\gamma_i)=\delta_i\] induces the isomorphism $\alpha:\Gamma_1\to\Gamma_2$. The corresponding boundary morphism is the continuation to the boundary of the map 
\[\phi:T_1\to T_2:x\mapsto\alpha(w_x)(\ell_2 f\ell_1^{-1}(\tilde x)),\]
where $\tilde x$ is the unique element in $\ell_1(P_1)$ for which there is an unique word $w_x$ in $\gamma_1,\gamma_1^{-1},\dots,\gamma_g,\gamma_g^{-1}$ such that $w_x(\tilde x)=x$. The map $\phi$ is a surjective isometry, and therefore the induced boundary morphism is absolutely continuous with respect to the Patterson--Sullivan  measures. Moreover, since $\phi$ sends orbits to orbits, it holds for all $R\geq 0$, that $n_Y(R)=n_{\phi(Y)}(R)$, and hence $e(\Gamma_1)=e(\Gamma_2)$.
\end{proof}

\section{Mumford curves}
We now turn our attention to the setting of algebraic curves. Let $k$ be a discrete valued field, with valuation $v_k$, and denote with $\mathcal O_k$ the ring of integers of $k$. Let $\pi$ be an uniformizer and let $\bar k$ be the residue field. The field $K$ is a complete algebraically closed field extension of $k$. 

The group $G=\PGL(2,k)$ acts on $\PP^1(K)$ through linear fractional transformations:
\[\left(\begin{array}{cc}a&b\\c&d\end{array}\right)\cdot z=\frac{az+b}{cz+d},\]
leaving the space of  $k$-rational points $\PP^1(k)$ invariant. A group element $g\in G$ is called \emph{hyperbolic} if the two eigenvalues have different valuation. Hyperbolic group elements have two fixed points, both lying in $\PP^1(k)$.

\subsection*{Schottky groups}
Analogous to tree-Schottky groups we will define Schottky groups, which are certain subgroups of $\PGL(2,k)$.

\begin{definition}
 A subgroup $\Gamma$ of $\PGL(2,k)$ is called a Schottky group if 
\begin{enumerate}
 \item $\Gamma$ is finitely generated; 
 \item all non-identity elements $\gamma\in\Gamma$ are hyperbolic.
\end{enumerate}
\end{definition}
\begin{definition}
The quotient space $X_\Gamma=\Gamma\backslash(\PP^1(K)-\Lambda_\Gamma)$ is a well-defined one-dimensional compact rigid analytic space, and as such, it is the unique analytification of an algebraic curve over $K$. We will not distinguish the analytic and algebraic curve, and call it simply a \emph{Mumford curve}.
\end{definition}

\subsection*{The tree $T(\Lambda_\Gamma)$}

There is a tree $T(\Lambda_\Gamma)$ related to $\Lambda_\Gamma$ in a natural way, cf.\ (\cite{GvdP}, chapter 1, \S 2). First, we consider the Bruhat-Tits tree $\T$ of $\PGL(2,k)$. This is a $q+1$-regular tree (with $q$ the cardinality of the residue field of $k$), which has a natural action by $\PGL(2,k)$ induced by its interpretation as equivalence classes of rank two lattices over the ring of integers of $k$. In this way, the boundary of the Bruhat-Tits tree also becomes naturally identified with $\PP^1(k)$. 

Given a Schottky group $\Gamma$, we let $T(\Lambda_\Gamma)$ denote the subtree of $\T$ \emph{spanned} by the limit set in the following sense: the subtree is the union of all vertices and edges that occur in geodesics $[\xi_1,\xi_2]
$ with $\xi_1 \in \Lambda_\Gamma$ and $\xi_2 \in \Lambda_\Gamma$. A map on the boundary that preserves the cross ratio (in the sense of formula (\ref{crossratio})) defines an isometry of the tree by the construction from Figure \ref{isometryF}. Conversely, an isometry of the tree preserves the cross ratio by formula (\ref{crossequal}). By construction, a matrix $\gamma \in \Gamma \subseteq \PGL(2,k)$ acts on the Bruhat-Tits tree by isometries (and its subtree $T(\Lambda_\Gamma)$), and this induces an action of $\gamma$ on the boundary, which is the natural action on $\PP^1(k)$.    

Since the action of the group $\Gamma$ is isometric, it is a tree-Schottky group for this action. The boundary of the tree $T(\Lambda_\Gamma)$ is $\Lambda_\Gamma$ by construction. The action of $\Gamma$ on $T(\Lambda_\Gamma)$ captures information about the curve $T(\Lambda_\Gamma)$ in the following sense. The quotient graph $\Gamma\backslash T(\Lambda_\Gamma)$ is the intersection graph of the special fibre of $X_\Gamma$, (\cite{Mumford} \S 3, \cite{GvdP}, remark 2.12.3). In particular, the quotient is a finite graph.

\subsection*{The curve and the tree}
As stated above, the curve $X_\Gamma$ and the tree $T(\Lambda_\Gamma)$ are strongly related. Another instance of this was pointed out by Teitelbaum \cite{Teitelbaum}: the vector spaces of rigid analytic modular forms on $\PP^1(K)-\Lambda_{\Gamma}$ are isomorphic to spaces of harmonic cocycles on the edges of the tree $T(\Lambda_\Gamma)$. We recall the notions involved and then describe the isomorphism.

\begin{definition}
Let $n$ be an even integer. Denote by $\Sp_n(\Gamma)$ the space of \emph{rigid analytic modular forms} on $\PP^1(K)-\Lambda_{\Gamma}$ of weight $n$ for $\Gamma$ over $K$, i.e., $\Sp_n(\Gamma)$ comprises the rigid analytic functions $f:\PP^1(K)-\Lambda_{\Gamma}\to K$ such that 
\[f(\gamma z)=\frac{(cz+d)^n}{(\det(\gamma))^{n/2}}f(z).\]
\end{definition}

\begin{remark}
The space of weight two rigid analytic modular forms for a Schottky group $\Gamma$ is isomorphic to the space of holomorphic 2-forms on the corresponding Mumford curve: 
$$ \Sp_2(\Gamma) \to H^0(X_\Gamma,\Omega_{X_\Gamma}) $$
by mapping a modular form $f$ to the differential form corresponding to $f(z)dz$ on $\PP^1(K)-\Lambda_{\Gamma}$. More generally, the space of weight $2\ell$ analytic modular forms is isomorphic to the space of holomorphic $\ell$-forms:
\[\Sp_{2\ell}(\Gamma) \to H^0(X_\Gamma,\Omega_{X_\Gamma}^{\otimes \ell}),\]
with the isomorphism induced by $f \mapsto f dz^{\otimes \ell}$.  
\end{remark}
 
\begin{definition} Let $n$ be an positive even integer. 
Denote with $P_{n}[X]\subset K[X]$ the vector space of polynomials of degree at most $n$. 
A $P_n[X]$-valued function $c$ on the edges of the tree $T(\Lambda_\Gamma)$ is called a \emph{$\Gamma$-harmonic cocycle} if
\begin{enumerate}
\item for all edges $e$, $c(e)=-c(\bar e)$, where $\bar e$ is the inverted orientation of $e$.
\item for all vertices $v$, $\sum_{e\mapsto v}c(e)=0$, where the sum is over the oriented edges with target $v$. \label{somisnul}
\item $c$ is equivariant with respect to the $\Gamma$-action, i.e., 
\[c(\gamma e)(X)=\gamma\cdot c(e)(X).\] 
Here $\gamma\in\Gamma$ acts on $p\in P_n[X]$ by:
\[\gamma\cdot p(X)=\frac{(bX+d)^n}{(\det(\gamma))^{n/2}}\cdot p(\frac{aX+c}{bX+d}).\]
\end{enumerate}
The set of harmonic $\Gamma$-cocycles with values in $P_{n-2}[X]$ is a vector space over $K$ denoted by $\Coc(\Gamma,n)$.
\end{definition}

In \cite{Schneider} (page 228), Peter Schneider constructed distributions on the boundary of the tree based on these harmonic cocycles. Recall that a distribution is a finitely additive function on the compact opens. Let $c\in\Coc(\Gamma,2)$. Then the corresponding distribution $\mu_c$ is defined by the following fundamental relation:
\begin{eqnarray*}
\mu_c(U(e))&=&c(e),
\end{eqnarray*}
where $U(e)$ is the compact open corresponding to the basic set $A(e)$, as in Figure \ref{AeUe}. 

These distributions are only finitely additive, because no convergence over infinite sums can be guaranteed. They can be seen as measures of finite unions of sets of the form $U(e)$. In particular, by property (\ref{somisnul}), the ``measure'' of the entire space is $\mu_c(\Lambda_\Gamma)=0$. It was found by Teitelbaum that one can use the integration theory of Amice-V\'elu and Vishik to integrate with respect to such distributions, and to construct a Poisson kernel for modular forms: 

\begin{theorem}[Teitelbaum \cite{Teitelbaum}]\label{poissonkernel}
For any Schottky group $\Gamma$ in $\PGL(2,k)$ and any even positive integer $n$, there are explicit back and forth isomorphisms $(\mathrm{Poisson}(n),\Res(n))$:
\[\xymatrix{ \Coc(\Gamma,n)  \ar@/^1.5pc/[r]^{\mathrm{Poisson(n)}}   & H^0(X_\Gamma,\Omega_{X_\Gamma}^{\otimes n/2})   \ar@/^1.5pc/[l]^{\mathrm{Res}(n)}  }\]
of $K$-vector spaces.
\end{theorem}

The maps are described explicitly as follows: in (\cite{Schneider}, p.\ 221, \cite{Teitelbaum} p.\ 397) Schneider and Teitelbaum define:
\[\Sp_{2n}(\Gamma)\to \Coc(\Gamma,2n) \colon f \mapsto c_f \]
with \[c_f(e)(X)=\sum_{i=0}^{2n-2}\Res_e(z^if dz)\binom{2n-2}{i} X^i,\]
where the \emph{residue} $\Res_e(\omega)$ is the usual rigid analytic residue of a differential form along the annulus defined by the edge $e$. 
The inverse of this mapping is the Poisson integral of Teitelbaum in (\cite{Teitelbaum}, part 2):
\[f(z)=\int_{\Lambda_\Gamma}\frac{1}{(z-x)}d\mu_c(x).\]

We finish this section with the proof of the first part of our theorem.

\begin{proof}[Proof of Theorem \ref{mainthm}, part \textup{(i)}]
Let $\Gamma_1$ and $\Gamma_2$ be Schottky groups of the same genus. Let $\alpha:\Gamma_1\isomto\Gamma_2$ be a group isomorphism. By Corollary \ref{phiexists} there is an $\alpha$-equivariant boundary morphism
\[\f:\Lambda_{\Gamma_1}\to\Lambda_{\Gamma_2}.\]
By Theorem \ref{treerigidity2}, $\f$ is absolutely continuous with respect to the respective Patterson--Sullivan measures if and only if the special fibers, that is the quotient graphs $\Gamma_1\backslash T(\Lambda_{\Gamma_1})$ and $\Gamma_2\backslash T(\Lambda_{\Gamma_2})$, are isomorphic. Moreover, there is then an isometry $F:T(\Lambda_{\Gamma_1})\to T(\Lambda_{\Gamma_2})$. This isometry induces a $K$-linear isomorphism by pullback via $F^{-1}$: 
\[F_*: \Coc(\Gamma_1, 2)\to\Coc(\Gamma_2, 2), c\mapsto c\circ F^{-1}.\]
\end{proof}

\begin{remark}
It is \emph{not} true that pullback by $F^{-1}$ induces an isomorphism of higher weight harmonic cocycles for $\Gamma_1$ and $\Gamma_2$ --- the problem is in the equivariance in general weight $\ell$, which means invariance in weight $2$. As we will see in the next section, this is exactly the key to formulating a measure-theoretic criterion for isomorphism. 
\end{remark}

\section{Rigidity for Mumford curves} 
An important consequence of the theorem of Teitelbaum is that we can define a \emph{product} of harmonic cocycles, simply by postulating the commutativity of the diagram  (\ref{dia}) from the introduction. This multiplication can be used to translate multiplicative properties of modular forms / differential forms into those of harmonic measures. 

The multiplicative properties of differential forms and how they relate to the isomorphism type of the curve are the subject of classical theorems of Noether, Babbage \cite{Babbage}, Enriques-Chisini and Petri \cite{Petri} (see \cite{Dodane} for a detailed historical overview). We refer to Saint-Donat \cite{SaintDonat} or \cite{ACGH} III.3 for a modern proof over arbitrary algebraically closed ground fields (which is the setting that we will need). 

Suppose now that $X=X_\Gamma$ is not hyperelliptic. Then $\Omega_{X}$ is normally generated, by a theorem of Noether, so the natural multiplication map 
$$ \mu_n \colon \mathrm{Sym}^n H^0(X,\Omega_X) \rightarrow H^0(X,\Omega_X^{\otimes n}) \colon f_1 \otimes \dots \otimes f_k \mapsto f_1 \cdot \dots \cdot f_k $$
is surjective for all $n \geq 1$;  $\Omega_X$ is base point free and we have the canonical embedding $$\varphi_X \colon X \rightarrow \PP|\Omega_X|=\PP^{g-1}.$$ Now $\ker(\mu_n)$ (for varying $n$) generate the ideal $I_{\Omega_X}$ that defines a canonical embedding of $X$, and the theorem of Babbage--Chisini--Enriques--Petri and Saint-Donat says the following: 

\begin{theorem}\label{bceps}For non-hyperelliptic $X$, the ideal  $I_{\Omega_X}$ is spanned by $\ker(\mu_n)$ for $n=2$, unless one of the following cases occurs:  \begin{enumerate} \item[\textup{(a)}] if $g=3$, the ideal is generated by the quartic that generates $\ker(\mu_4)$; \item[\textup{(b)}] if $g=6$ and the curve is isomorphic to a plane quintic, the ideal is generated by $\ker(\mu_n)$ for $n \leq 3$, but is not generated by $\ker(\mu_2)$ only; \item[\textup{(c)}] if $X$ is trigonal (i.e., admits a degree 3 morphism to $\PP^1$), then the ideal is generated by $\ker(\mu_n)$ for $n \leq 3$, but is not generated by $\ker(\mu_2)$ only. \qed \end{enumerate}
\end{theorem}

\begin{proof}[Proof of Theorem \ref{mainthm}, part \textup{(ii)}]
Suppose that we have two Schottky groups $\Gamma_1$ and $\Gamma_2$, corresponding to non-hyperelliptic Mumford curves $X_1$ and $X_2$. From part (i) of the theorem, we have a $K$-linear isomorphism of vector spaces $$F_* \colon \Coc(\Gamma_1,2) \rightarrow \Coc(\Gamma_2,2).$$  
 Recalling the isomorphisms of $K$-vector spaces $$H^0(X_1,\Omega_{X_1}) \cong \Sp_2(\Gamma) \cong \Coc(\Gamma,2)$$ for $\Gamma=\Gamma_i$, the map $F_*$ can be considered as a linear isomorphism 
 $$H^0(X_1,\Omega_{X_1}) \isomto H^0(X_2,\Omega_{X_2}),$$ which extends to a map of the symmetric algebra 
 $$ F_* \colon \bigoplus_{n \geq 0} \mathrm{Sym}^n H^0(X_1,\Omega_{X_1}) \isomto \bigoplus_{n \geq 0} \mathrm{Sym}^n  H^0(X_2,\Omega_{X_2}), $$
and hence to an isomorphism $$\PP|\Omega_{X_1}| \rightarrow \PP|\Omega_{X_2}|.$$ The map induces an isomorphism between $X_1$ and $X_2$ precisely if it respects the ideal of relations in the canonical embeddings, i.e., if $F_*(\ker(\mu^1_n)) \subseteq \ker(\mu^2_n)$ for all weights $n$. 
 
 Using the Poisson kernel Theorem \ref{poissonkernel} of Teitelbaum, we can translate this condition to a condition on harmonic measures, as follows: 
for any $n$ and $\Gamma=\Gamma_i$, consider the map 
 $$ \mu_n'(\Gamma) \colon \mathrm{Sym}^n \Coc(\Gamma,2) \overset{\mathrm{Poisson}(2)^{\otimes n}}{\longrightarrow} \mathrm{Sym}^n H^0(X,\Omega_X) \overset{\mu_n}{\rightarrow} H^0(X,\Omega_X^{\otimes n}) \overset{\mathrm{Res(n)}}{\longrightarrow} \Coc(\Gamma,2n).  $$
 Then the isomorphism of $X_1$ and $X_2$ is equivalent to 
 $$ F_*(\ker(\mu_n'(\Gamma_1))) = \ker(\mu_n'(\Gamma_2))\mbox{ for all } n \geq 0. $$
 Now Theorem \ref{bceps} implies that we only need to require this for the given $n$ in the given cases. Since $X_1$ an $X_2$ are isomorphic over $K$ and defined by Schottky groups $\Gamma_1,\Gamma_2$ in $\PGL(2,k)$, the groups $\Gamma_1$ and $\Gamma_2$ are conjugate in $\PGL(2,k)$ (Mumford, \cite{Mumford} Corollary 4.11).
 \end{proof}
 
 \begin{remark}
 We do not know how to deal with the case where the curves are hyperelliptic. The reason for this is that the measure-theoretic property from tree-rigidity only gives a natural isomorphism on the space of harmonic cocycles of weight two, but for a hyperelliptic curve, the image of the canonical map is $\PP^1$. Trying to use a pluricanonical embedding does not help, since we cannot conclude from tree-rigidity that the tree isometry induces an isomorphism of the corresponding spaces of higher weight harmonic measures. 
 \end{remark}
 
 \begin{remark}
 One may try to translate the property of being trigonal into a property of the reduction graph or the Schottky group. 
 Matthew Baker \cite{Baker} has introduced and studied the concept of gonality of a graph, and his results imply in our case that the gonality of the reduction graph of $X$ (for a semistable model in which all  components are smooth) bounds above the $K$-gonality of $X$.
 \end{remark}
 
 \begin{remark}
 One may wonder how constructive this theorem can be made, in the following sense. Suppose given two Schottky groups of the same genus, given by finite sets of explicit generating matrices in $\PGL(2,k)$. Are the corresponding curves isomorphic? From the point of view of traditional rigidity, the question is equivalent to the Schottky group being conjugate in $\PGL(2,k)$; but we do not know an efficient algorithm for $k$ an infinite (recursive) field, and the general subgroup conjugacy problem is believed to be computationally hard (compare e.g., \cite{RoneyDougal}). From our point of view of measure theoretic rigidity, one would have to construct a basis for the spaces of harmonic measures, find the matrix that represents $F$ on them, and then compute kernels of multiplication maps. The computability of the latter depends on the computability of the maps $(\mathrm{Poisson},\mathrm{Res})$. 
 \end{remark}

\begin{remark} \label{ak}
It is well-known that two sets of four distinct points in $\PP^1(k)$ can be mapped to each other by a fractional linear transformation if and only if they have the same cross ratio.  Now given two sets $S$, $Q$ of $2g$ ($g \geq 2$) distinct points in $\PP^1(k)$, is it possible to find an element $g\in\PGL(2,k)$ such that $g(S)=Q$? Our theorem relates to this problem in the following way. Make a partition $\mathrm{P}_S$ of the set $S$ into pairs $(s_i^+,s_i^-)$ for $1\leq i\leq g $. For every pair, choose a hyperbolic element $\gamma_i\in\PGL(2,k)$ of amplitude $1$ and for which $s_i^+$ (respectively $s_i^-$) is the attracting (respectively the repelling) fixed point of $\gamma_i$. Choose minimal exponents $(n_1,...,n_g)$ such that 
\[\Gamma(\mathrm {P}_S)=\langle\gamma_1^{n_1},...,\gamma_g^{n_g}\rangle,\]
is a Schottky group (to find these minimal exponents, one may for example use Remark \ref{funddomain}). Since all data needed is  of purely metric nature on the tree, these minimal exponents are invariant under conjugation. Now the following holds: for $S$ and $Q$, there is an element $g\in\PGL(2,k)$ such that $g(S)=Q$ if and only if there is a partition $\mathrm{P}_S$ and a partition $\mathrm{P}_Q$ such that  
$g\Gamma(\mathrm{P}_S)g^{-1}=\Gamma(\mathrm{P}_Q),$ which means
$$ X_{\Gamma(\mathrm{P}_S)} \cong X_{\Gamma(\mathrm{P}_Q)}. $$
For this, our main theorem gives a purely measure theoretic reformulation (to be checked on the finitely many choices of partitions). 
\end{remark}

\bibliographystyle{amsplain}

\end{document}